\documentclass[11pt]{article}

\usepackage{amsmath, amsthm, amssymb, color, enumitem, graphicx, hyperref, tikz}
\usepackage[a4paper]{geometry}
\usepackage[nobysame, initials]{amsrefs}

\usetikzlibrary{3d}
\def\endofClaim{\hfill\scalebox{.6}{$\Box$}}

\def\l{\ell}

\def\phi{\varphi}
\def\le{\leqslant}
\def\ge{\geqslant}

\def\RR{\mathbb{R}}

\newtheorem{theorem}{Theorem}
\newtheorem{lemma}[theorem]{Lemma}

\newtheorem{prop}[theorem]{Proposition}

\newtheorem{conj}[theorem]{Conjecture}

\theoremstyle{definition}

\theoremstyle{remark}

\newcommand{\oldqed}{}

\title{Equilateral sets in the $\l_1$ sum of Euclidean spaces}
\author{Aaron Lin\footnote{Department of Mathematics, London School of Economics and Political Science, London, United Kingdom.}}
\date{}

\begin{document}

\maketitle

\begin{abstract}
Let $E^n$ denote the (real) $n$-dimensional Euclidean space. It is not known whether an equilateral set in the $\l_1$ sum of $E^a$ and $E^b$, denoted here as $E^a \oplus_1 E^b$, has maximum size at least $\dim(E^a \oplus_1 E^b) + 1 = a + b + 1$ for all pairs of $a$ and $b$. We show, via some explicit constructions of equilateral sets, that this holds for all $a \le 27$, as well as some other instances.
\end{abstract}


\section{The Problem}\label{sec:problem}

An equilateral set in a normed space $(X, \lVert \cdot \rVert)$ is a subset $S \subset X$ such that given a fixed $\lambda > 0$, we have $\lVert x -y \rVert = \lambda$ for all distinct $x,y \in S$. Since norms respect scalar multiplication, the maximum size of an equilateral set in a normed space $X$ is well-defined, and we denote it by $e(X)$. When $\dim(X) = n$, we have the tight upper bound $e(X) \le 2^n$, proved in \cite{Petty} by Petty over 40 years ago. However, the following conjecture concerning a lower bound on $e(X)$, formulated also by Petty (amongst others), remains open for $n \ge 5$. (The $n=2$ case is easy; see \cite{Petty, Vaisala} and \cite{Makeev} for the $n = 3$ and $4$ cases respectively.)


\begin{conj}\label{conj:lower_bound}
Let $X$ be an $n$-dimensional normed space. Then $e(X) \ge n + 1$.
\end{conj}


We wish to verify this conjecture for the Cartesian product $\RR^a \times \RR^b$, equipped with the norm $\lVert \cdot \rVert$ given by
\begin{equation*}
\lVert (x, y) \rVert = \lVert x \rVert_2 + \lVert y \rVert_2,
\end{equation*}
where $x \in \RR^a$, $y \in \RR^b$, and $\lVert \cdot \rVert_2$ denotes the Euclidean norm. We denote this space by $E^a \oplus_1 E^b$, and refer to it as the $\l_1$ sum of the Euclidean spaces $E^a$ and $E^b$. This was considered originally by Roman Karasev of the Moscow Institute of Physics and Technology, as a possible counterexample to Conjecture \ref{conj:lower_bound}. See \cite[Section 3]{Swanepoel} for more background on equilateral sets.


\section{The Results}\label{sec:result}

Observe that we need only construct $a + b + 1$ points in $E^a \oplus_1 E^b$ which form an equilateral set to show that $e(E^a \oplus_1 E^b) \ge \dim(E^a \oplus_1 E^b) + 1 = a + b + 1$. We will work with these points in the form $(x_i, y_i) \in \RR^a \times \RR^b$, since we can then examine the $x_i$'s and $y_i$'s separately when necessary. By abuse of notation, we will denote the origin of any Euclidean space by $o$.

Let $d_n$ denote the circumradius of a regular $n$-simplex ($n \ge 1$) with unit side length. Note that
\begin{equation*}
d_n = \left( \sqrt{2+\frac{2}{n}} \right)^{-1}
\end{equation*}
is a strictly increasing function of $n$, and we have $1/2 \le d_n < 1/\sqrt{2}$.


The $a = 1$ case is easy.

\begin{prop}\label{prop:a=1}
$e(E^1 \oplus_1 E^b) \ge b + 2$.
\end{prop}

\begin{proof}
Let $y_1, \dotsc, y_{b+1}$ be the vertices of a regular $b$-simplex with unit side length centred on the origin. Then the points $(o, y_1), \dotsc, (o, y_{b+1}), (1-d_b, o)$ are pairwise equidistant. 
\end{proof}


We next deal with the case where $b = a$.

\begin{prop}\label{prop:a=b}
$e(E^a \oplus_1 E^a) \ge 2a + 1$.
\end{prop}

\begin{proof}
We first describe an equilateral set of size $2a$ in $E^a \oplus_1 E^a$: 
consider the set of points $\{ (v_i, \frac{1}{2}e_i) : i = 1, \dotsc, a\} \cup \{ (v_i, -\frac{1}{2}e_i) : i = 1, \dotsc, a\}$, where $v_1, \dotsc, v_a$ are the vertices of a regular simplex of codimension one, centred on the origin with side length $1 - 1/\sqrt{2}$, and $e_1, \dotsc, e_a$ are the standard basis vectors. Note that the $2a$ vectors $\pm \frac{1}{2} e_i$ for $i = 1, \dotsc, a$ form a cross-polytope in $E^a$, centred on the origin. 



We now want to add a point of the form $(x, o)$ to the above set, a unit distance away from every other point. 
Note that we must have $\lVert x - v_i \rVert_2 = 1/2$ for $i = 1, \dotsc, a$, and
$x$ must lie on the one-dimensional subspace
orthogonal to the $(a-1)$-dimensional subspace spanned by the $v_i$'s. 
This is realisable if $\lVert x - v_i \rVert_2 \ge (1 - 1/\sqrt{2})d_{a-1}$ (note that the $(a-1)$-simplex formed by the $v_i$'s has side length $1 - 1/\sqrt{2}$),
in which case we have an equilateral set of size $2a + 1$ in $E^a \oplus_1 E^a$. 
But we have
\begin{equation*}
\frac{1}{2} > \frac{1}{\sqrt{2}} \left( 1 - \frac{1}{\sqrt{2}} \right) > \left( 1 - \frac{1}{\sqrt{2}} \right)d_{a-1} 
\end{equation*}
for all $a \ge 2$.
\end{proof}


In the remaining case and our main result, we have $b > a \ge 2$, and we find sufficient conditions for an equilateral set of size $a + b + 1$ to exist in $E^a \oplus_1 E^b$. 

\begin{theorem}\label{thm}
Let $b > a \ge 2$. Let $c = \lfloor 1 + b/(a + 1) \rfloor$, $\beta = b \pmod{a+1} \in \{0, \dotsc, a\}$, and $\alpha = a + 1 - \beta$. If $\beta = 0$, $1$, or $a$, or 
\begin{equation}\label{eqn:hard_cond_ugly}
\frac{\alpha - 1}{2\alpha} \left( 1 - \sqrt{\frac{c-1}{c}} \right)^2 + \frac{\beta - 1}{2\beta} \left( 1 - \sqrt{\frac{c}{c+1}} \right)^2 \le
\left( 1 - \sqrt{\frac{1}{2} \left( \frac{c-1}{c} + \frac{c}{c+1} \right)} \right)^2
\end{equation}
holds, then $e(E^a \oplus_1 E^b) \ge a + b + 1$.
\end{theorem}


Note that if inequality (\ref{eqn:hard_cond_ugly}) is satisfied by all pairs of $a$ and $b$ with $b > a \ge 2$ and $b \ne 0$, $1$, or $a \pmod{a+1}$, then Proposition \ref{prop:a=1}, Proposition \ref{prop:a=b}, and Theorem \ref{thm} cover all possible cases, as $E^a \oplus_1 E^b$ is isometrically isomorphic to $E^b \oplus_1 E^a$. Unfortunately, this is not true, and we explore its limitations after the proof of Theorem \ref{thm}.


\begin{proof}[Proof of Theorem \ref{thm}]
We are going to describe an equilateral set of size $a + b + 1$ with unit distances between points.
Noting that $\alpha \cdot (c-1) + \beta \cdot c = b$, consider the following decomposition of $E^b$ into pairwise orthogonal subspaces:
\begin{equation*}
E^b = U_1 \oplus \dotsb U_\alpha \oplus V_1 \oplus \dotsb \oplus V_\beta,
\end{equation*}
where $\dim U_i = c - 1$ for $i = 1, \dotsc, \alpha$ and $\dim V_j = c$ for $j = 1, \dotsc, \beta$.
Let $u_1^{(i)}, \dotsc, u_c^{(i)}$ be the vertices of a regular $(c-1)$-simplex with unit side length centred on the origin in $U_i$, and let $v_1^{(j)}, \dotsc, v_{c+1}^{(j)}$ be the vertices of a regular $c$-simplex with unit side length centred on the origin in $V_j$.

The $a + b + 1$ points of our equilateral set will be
\begin{equation*}
\left \{ \left(w_i, u_k^{(i)} \right) : 1 \le i \le \alpha, 1 \le k \le c \right \} \cup \left \{ \left(z_j, v_\ell^{(j)} \right) : 1 \le j \le \beta, 1 \le \ell \le c + 1 \right \}.
\end{equation*}
Note here that $\alpha \cdot c + \beta \cdot (c+1) = a + b + 1$, and we have $\lVert u_k^{(i)} - u_{k'}^{(i)} \rVert_2 = \lVert v_\ell^{(j)} - u_{\ell'}^{(j)} \rVert_2 = 1$ for $k \ne k'$ and $\ell \ne \ell'$.
All that remains is then to calculate how far apart the $w_i$'s and $z_j$'s should be in $E^a$, and see if such a configuration is realisable.



We only have three non-trivial distances to calculate:
\begin{itemize}
\item the distance between $\left(z_j, v_\ell^{(j)} \right)$ and $\left(z_{j'}, v_{\ell'}^{(j')} \right)$ for $j \ne j'$ should be one, and so
\begin{equation*}
\lVert z_j - z_{j'} \rVert_2 = 1 - \sqrt{d_{c}^2 + d_{c}^2} =  1 - \sqrt{\frac{c}{c+1}} =: f(c),
\end{equation*}
\item the distance between $\left(w_i, u_k^{(i)} \right)$ and $\left(w_{i'}, u_{k'}^{(i')} \right)$ for $i \ne i'$ should be one, and so
\begin{equation*}
\lVert w_i - w_{i'} \rVert_2 = 1 - \sqrt{d_{c-1}^2 + d_{c-1}^2} =  1 - \sqrt{\frac{c-1}{c}} = f(c-1),
\end{equation*}
\item finally, the distance between $\left(w_i, u_k^{(i)} \right)$ and $\left(z_{j}, v_{\ell}^{(j)} \right)$ should also be one, and so
\begin{equation*}
\lVert w_i - z_{j} \rVert_2 = 1 - \sqrt{d_{c-1}^2 + d_c^2} = 1 - \sqrt{\frac{1}{2} \left( \frac{c-1}{c} + \frac{c}{c+1} \right)} =: g(c).
\end{equation*}
\end{itemize}
What we need in $E^a$ is thus a regular $(\alpha - 1)$-simplex with side length $f(c-1)$ and a regular $(\beta - 1)$-simplex with side length $f(c)$, with the distance between any point from one simplex and any point from the other being $g(c)$. Note that here we consider the $(-1)$-simplex to be empty. We now show that this configuration is realisable (in $E^a$) if the conditions in the statement of the theorem are satisfied.



We first consider the special cases $\beta = 0$ and $\beta = 1$ or $a$, and then the main case $2 \le \beta \le a-1$. 
It is trivial if $\beta = 0$: then $\alpha = a + 1$ and we only need to find a regular $a$-simplex with side length $f(c-1)$ in $E^a$.

If $\beta = 1$, in which case $\alpha = a$, consider the decomposition $E^a = E^{a-1} \oplus E^1$. Consider the points $(p_1, o), \dotsc, (p_a, o)$, where $p_1, \dotsc, p_a$ are the vertices of a regular $(a-1)$-simplex with side length $f(c-1)$, centred on the origin in $E^{a-1}$. We want to add a point $(o, \zeta)$ for some $\zeta \in E^1$ such that, for any $i = 1, \dotsc, a$, we have
\begin{equation*}
\lVert (p_i, o) - (o, \zeta) \rVert_2 = g(c),
\end{equation*}
or equivalently,
\begin{equation*}
d_{a-1}^2 f(c-1)^2 + \zeta^2 = g(c)^2.
\end{equation*}
Noting that $d_{a-1} < 1/\sqrt{2}$, it suffices to show, for all $c \ge 2$, that
\begin{equation*}
f(c-1)^2 < 2 g(c)^2.
\end{equation*}
But this is easily verifiable to be true, and so the desired $a$-simplex exists in $E^a$.
By symmetry and the fact that $f(c)^2 < f(c-1)^2$, the desired $a$-simplex also exists if $\beta = a$.

Now suppose $2 \le \beta \le a - 1$ so that $\alpha, \beta \ge 2$. 
Consider this time, the decomposition $E^a = E^{\alpha - 1} \oplus E^{\beta - 1} \oplus E^1$, noting that $\alpha + \beta = a + 1$. Suppose $p_1, \dotsc, p_\alpha$ are the vertices of a regular $(\alpha-1)$-simplex with side length $f(c-1)$, centred on the origin in $E^{\alpha-1}$, and $q_1, \dotsc, q_\beta$ are the vertices of a regular $(\beta-1)$-simplex with side length $f(c)$, centred on the origin in $E^{\beta-1}$. Consider then the set of points $\{ (p_i, o, o) : i = 1, \dotsc, \alpha \} \cup \{ (o, q_j, \zeta) : j = 1, \dotsc, \beta \}$, where $\zeta \in E^1$ is to be determined.
As before, we want a $\zeta$ such that for all $i$ and $j$, we have
\begin{equation*}
\lVert (p_i, o, o) - (o, q_j, \zeta) \rVert_2 = g(c),
\end{equation*}
or equivalently
\begin{equation}\label{eqn:hard_cond_nice}
\left( d_{\alpha-1} f(c-1) \right)^2 + \left( d_{\beta-1} f(c) \right)^2 \le g(c)^2.
\end{equation}
But this is exactly inequality (\ref{eqn:hard_cond_ugly}).
\end{proof}


As mentioned above, inequality (\ref{eqn:hard_cond_ugly}), and thus inequality (\ref{eqn:hard_cond_nice}), does not hold for all pairs of $a$ and $b$. However, we have the following result.

\begin{lemma}\label{lem}
If $b \ge a ^2 + a$, then inequality (\ref{eqn:hard_cond_nice}) holds.
\end{lemma}

\begin{proof}
Since $f(n)$ is a decreasing function of $n$, inequality (\ref{eqn:hard_cond_nice}) holds if $a$ and $b$ satisfy
\begin{equation*}
\left( d_{\alpha-1}^2 + d_{\beta-1}^2 \right) f(c-1)^2 < g(c)^2.
\end{equation*}
Using the fact that $\alpha = a + 1 - \beta$ implies $d_{\alpha-1}^2 + d_{\beta-1}^2 \le (a-1)/(a+1)$, we therefore just need $a$ and $b$ to satisfy
\begin{equation*}
\frac{a-1}{a+1} < \left( \frac{g(c)}{f(c-1)} \right)^2.
\end{equation*}
But the latter expression is an increasing function of $c$, and so if $c \ge a$, or equivalently, when $b \ge a^2 + a$, we need only consider the inequality
\begin{equation*}
\frac{a-1}{a+1} < \left( \frac{g(a)}{f(a-1)} \right)^2,
\end{equation*}
which is then easily verifiable to be true. 
\end{proof}



It can be checked (by computer) that inequality (\ref{eqn:hard_cond_nice}) holds for all $a \le 27$, but does not hold for $a = 28$ and $b = 40$, $a = 29$ and $39 \le b \le 44$, and $a = 30$ and $40 \le b \le 47$. 
The spaces of smallest dimension where we could not find an equilateral set of size $a + b + 1$ are $E^{28} \oplus_1 E^{40}$ and $E^{29} \oplus_1 E^{39}$.



\subsection*{Acknowledgements}
The author would like to thank Konrad Swanepoel for introducing him to this problem, and for the numerous helpful suggestions in writing this up.

\vfill



\end{document}